\documentclass{article}

\usepackage{tikz-cd}

\usepackage{graphicx, amsmath, amssymb, longtable, mathtools, amsthm, polynom, mathrsfs, cancel}

\usepackage{url}
\usepackage{comment}
\usepackage[table]{xcolor}
\usepackage{booktabs} 

\usepackage[backend=biber,style=numeric]{biblatex}
\usepackage[breaklinks]{hyperref}
\appto\bibfont{\setlength{\emergencystretch}{.1em}}

\newtheorem{theorem}{Theorem} 
\numberwithin{theorem}{section}

\newtheorem{lemma}{Lemma}
\numberwithin{lemma}{section}

\newtheorem{definition}{Definition}
\numberwithin{definition}{section}

\numberwithin{assumption}{section}

\newtheorem{proposition}{Proposition}
\numberwithin{proposition}{section}

\numberwithin{remark}{section}

\newtheorem{corollary}{Corollary}
\numberwithin{corollary}{section}

\numberwithin{conjecture}{section}

\newtheorem{example}{Example}
\numberwithin{example}{section}

\usepackage{algorithm}
\usepackage{algpseudocode} 

\usepackage{multicol, multirow, array, tabularx, caption, float}
\newcolumntype{C}[1]{>{\centering\arraybackslash}p{#1}}
\usepackage{tikz}
\usetikzlibrary{positioning}
\usepackage{subcaption}

\usepackage[letterpaper, portrait, margin=1in]{geometry}
\usepackage{changepage}
\usepackage{array}

\newcommand{\R}{\mathbb{R}}

\newcommand{\N}{\mathbb{N}}

\newcommand{\E}{\mathbb{E}}

\newcommand{\bset}[1]{\left \lbrace #1 \right \rbrace}
\newcommand{\aset}[1]{\left \langle #1 \right \rangle}

\newcommand{\Span}[1]{\mathrm{Span}\bset{#1}}
\newcommand{\id}{\mathrm{id}}

\newcommand{\D}{\displaystyle}
\newcommand{\Cov}{\mathrm{Cov}} 
\newcommand{\intp}[1] {\D \int_\Omega {#1 } \; d \mathbb P}
\newcommand{\HS}{\mathrm{HS}}

\newcommand{\CH}{\mathcal H}

\title{Karhunen Lo\`eve Expansions of Hilbert Space-Valued Random Elements}
\author{Trajan Murphy}
\date{\today}

\begin{document}

\begin{multicols*}{2}
\maketitle

\begin{abstract}
    The Karhunen-Lo\`eve Expansion (KLE) of a stochastic process is a well understood eigenfunction expansion used widely in time series analysis, stochastic PDEs, and signal processing. Karhunen-Lo\`eve expansions have also been proven to exist for other types of stochastic elements whose values lie in certain $L^2$ spaces. This article provides a concise proof about the necessary and sufficient conditions for a function $v$ defined on some sample space $\Omega$ and whose values lie in some Hilbert space $\CH$ to admit an eigenfunction expansion like the well-known KLE. We draw on the existing theory of Bochner spaces and Hilbert-Schmidt spaces and construct an isomorphism between them. Furthermore, this isomorphism is natural, which has important computational consequences. Finally, we demonstrate with an example the computational advantages conferred by considering the KLE in this generalized setting. 
\end{abstract}

\section{Introduction}

The Karhunen-Lo\`eve expansion (KLE) of a stochastic process is a powerful theoretical tool with applications in signal processing, stochastic PDEs, and statistical machine learning. For a stochastic process $v = \bset{v_t}_{t \in \mathcal T}$, $\mathcal T \subseteq R$ a closed interval, it is well known (\cite{Aadithya2018,  Giambartolomei2015, alexanderian2015, chien1967}) that the KLE of $v$ is the infinite sum $v = \E(v_t) + \sum_{r=1}^\infty \lambda_r^{1/2} Y_r \phi_r$, with $Y_r$ the zero mean, unit variance, mutually uncorrelated stochastic components and $\phi_r$ the principal components. 

\parencite{steinwart2017} generalizes somewhat, allowing $\mathcal T$ to be any separable metric space and replacing $\mu$, the Lebesgue measure on $\R$, with a Borel measure $\nu$ on $\mathcal T$. 

In the framework of the above authors, a stochastic process can equivalently be thought of as a function $v: \mathcal T \to L^2(\Omega)$, where $(\Omega, \mathcal F, \mathbb P)$ is a suitable (complete) measure space, often a probability space, and $L^2(\Omega)$ comprises the Kolmogorov quotient space of real-valued, $\mathcal F$-measurable, square-integrable functions on $\Omega$. 

The implicit assumption in the existence of a KLE is that $\int_\mathcal T \int_\Omega |v_t(\omega)|^2 \; d \mathbb P(\omega) d \nu(t) < \infty$. By reversing the roles of the spatial variable $t$ and the stochastic variable $\omega$, we produce a function $v:\Omega \to L^2(\mathcal T)$. Our goal is to generalize the KLE by replacing $L^2(\mathcal T)$ with an arbitrary Hilbert space $\mathcal H$. This has been accomplished somewhat in \cite{castrilloncandas2022, betz2014} by replacing $L^2(\mathcal T)$ with the Hilbert space $L^2(U, \R^Q)$\footnote{The space $L^2(U, \R^Q)$ denotes set of Lebesgue measurable functions $v: U \to \R^Q$ such that $\int_U v(x)^\top v(x) d\mu(x) < \infty$. this space is equipped with the inner product $\aset{u,v}_{L^2(U, \R^Q)} := \int_U  u(x)^\top v(x) d \mu(x)$. Functions which agree $\mu$-a.e. are identified.}, where $U$ is a domain of $\R^n$. As before, we must assume $v: \Omega \to L^2(U, \R^Q)$ has the property that $\E(\|v(\omega)\|_{L^2(U, \R^Q)}^2) < \infty$, in which case we have the expansion

\begin{equation} \label{Preliminary KLE}
    v(\omega) = \E(v) + \sum_{r=1}^\infty \lambda_r^{1/2} Y_r(\omega) \phi_r
\end{equation}

Here $\bset{\lambda_r, \phi_r}_{r=1}^\infty$ constitute the descending eigenpairs of the \textit{covariance operator} $\mathscr K: L^2(U, \R^Q) \to L^2(U, \R^Q)$ given by $\mathscr Kf (x) = \int_U K(x,y) f(y) d \mu(y)$, $\mu$ the Lebesgue measure on $U$, where $K: U \times U \to \R$ is the \textit{covariance kernel} defined by $K(x,y) :=  \E[(v(\omega, x) - \E(v(\omega, x))^\top(v(\omega, y) - \E(v(\omega,y))]$.
$\bset{Y_r}_{r=1}^\infty$ again constitutes the sequence of zero-mean, unit-variance, mutually uncorrelated random variables, computed by $\lambda_r^{-1/2} \int_U (v(\omega,x) - \E(v(\omega,x)) \phi_r(x)\; d \mu(x)$. Note that the space $L^2(U, \R^Q)$ generalizes $L^2(\mathcal T)$ if we take $U = \mathcal T$ and $Q = 1$. \cite{levy1999} reports a generalization of expansion \eqref{Preliminary KLE}  for any Hilbert space of functions $U \to \R^Q$, extending the formulation to Sobolev spaces and spaces with weighted inner products.  

\parencite{Schwab2006} is, to current knowledge, the closest review to a KLE in a general Hilbert space, although there are a few deficiencies. First, each Hilbert space is assumed separable. While \cite{Schwab2006} presents a thorough development for the existence of a KL-like expansion of an element $v$ in the tensor product space of two arbitrary Hilbert spaces $\mathcal H$ and $\mathcal S$ (given as Thm. 2.5), the KLE of an element in $L^2(\Omega \times U)$ is merely left as a proposition (Prop. 2.8). While the isomorphism $L^2(\Omega \times U) \cong L^2(\Omega) \otimes L^2(U)$ is well understood when $L^2(\Omega)$ and $L^2(U)$ are both separable Hilbert spaces \cite{kolmogorov1961}, this correspondence is never explicitly cited in \parencite{Schwab2006}. \cite{picci26} also cites a KLE for the case where $\CH$ is separable. 

To suitably generalize the KLE we must introduce two spaces: $L^2(\Omega, \mathcal H)$ - a space of functions $v: \Omega \to \mathcal H$ and $\text{HS}(\mathcal H, L^2(\Omega))$ - a space of compact linear operators. The existence of a KLE for an arbitray $\mathcal H$-valued random element relies on a natural isomorphism between $L^2(\Omega, \mathcal H)$ and $\text{HS}(\mathcal H, L^2(\Omega))$, and the ability to perform singular value decompositions on operators in $\text{HS}(\mathcal H, L^2(\Omega))$.

\section{The KL Expansion In An Arbitrary Hilbert Space}
\subsection{The Space $L^2(\Omega, \mathcal H)$}

Let $(\Omega, \mathcal F, \mathbb P)$ a complete measure space and $\mathcal H$ a Banach space over $\R$ with $\mathcal H^*$ its continuous dual. If $\mathcal H$ is a Hilbert space, we identify $\mathcal H^*$ with $\mathcal H$. The following definitions are due to \cite{DiestelUhl1977}.

\begin{definition} [Simple Function] \label{Simple Definition}
    A function $v: \Omega \to \mathcal H$ is \textit{simple} if $\mathbb P$-a.e., $v(\omega) = \sum_{n=1}^N \mathbb I_{E_n}(\omega) e_n$, for $E_n \in \mathcal F$, $e_n \in \mathcal H$.
\end{definition}

\begin{definition} [Bochner Measurable]
 We say that $v$ is \textit{Bochner measurable} or \textit{strongly measurable} if there exists a sequence $\bset{v_n}_{n = 1}^\infty$ of simple functions such that $\lim_{n \to \infty} \|v_n(\omega) - v(\omega)\|_\mathcal H = 0$, $\mathbb P$-a.e.
\end{definition}

\begin{definition} [Weakly Measurable]
    We say that $v$ is weakly measurable if the scalar-valued mapping $\omega \mapsto x^*v(\omega)$ is Borel measurable for each continuous linear functional $x^* \in \mathcal H^*$ 
\end{definition}

\begin{definition} [Essentially Separably Valued] \label{Essentially separably valued}
    We say that $v$ is essentially separably valued (ESV), if there exists a null set $\Omega_0 \in \mathcal F$ such that $v(\Omega_0^C)$ is a separable subset of $\mathcal H$. 
\end{definition}

\begin{comment}
\begin{definition} [Countably Valued]
    We say that $v$ is countably valued if $v(\omega) = \sum_{n=1}^\infty \mathbb I_{E_n}(\omega) e_n$ for sequences $\bset{E_n}_{n = 1}^\infty \subseteq \mathcal F$ disjoint, $\bset{e_n}_{n=1}^\infty \in \mathcal H$. 
\end{definition}

As the $\bset{E_n}$ are assumed disjoint, the sum above only has one non-zero entry for each $\omega \in \Omega$, so concerns about convergence can be postponed.

\begin{definition} [$\Gamma$-measurable function]
    Let $\Gamma \subseteq \mathcal H^*$. We say that $v$ is $\Gamma$-measurable if $\omega \mapsto x^*v(\omega)$ is Borel-measurable for all $x^* \in \Gamma$. 
\end{definition} 

\begin{definition}[Norming] \label{Norming}
    We say that $\Gamma \subseteq \mathcal H^*$ is norming if for all $x \in \mathcal H$, we have 

    $$\|x\|_\mathcal H = \sup_{x^* \in \Gamma} \bset{|x^*x|/{\|x^*\|}}$$

    where we define $0/0 = 0$.
\end{definition}

\end{comment}

\begin{theorem} [Pettis' Theorem] \label{Pettis Theorem} 
    The following are equivalent 

    \begin{enumerate} 
        \item $v$ is Bochner measurable 

        \item $v$ is weakly measurable and essentially separably valued


    \end{enumerate}
\end{theorem}

The well known Pettis' Theorem can be found as Ch.2 Thm 2. of \cite{DiestelUhl1977}.



\subsection{The Expectation Operator}
\label{The Expectation Operator}

Suppose $v$ is Bochner measurable and $\bset{v_k}_{k=1}^\infty$ is a sequence of simple functions converging $\mathbb P$-a.e. to $v$. Since each $v_k$ is simple, we may write it as 

$$v_k = \sum_{n = 1}^{N_k} \mathbb I_{E^k_n}(\omega)e^n_k$$

The Bochner integral of a simple function over $E$ is defined to be the vector $\sum_{n = 1}^{N_k} \mathbb P(E^k_n \cap E)e^n_k$ for each measurable $E$ and is denoted $\int_E v_k(\omega) \; d \mathbb P(\omega)$ or simply $\int_E v_k \; d \mathbb P$. If it is also true that the scalar-valued integral $\int_E \|v_k - v\| d \mathbb P \to 0$ as $k \to \infty$ for each $E \in \mathcal F$, then we will say that $v$ is \textit{Bochner integrable} \cite{DiestelUhl1977}. The Bochner integral of $v$ over $E$ is then defined to be the limit as $k \to \infty$ of the sequence $\bset{ \int_E v_k \; d \mathbb P}_{k=1}^\infty$ \cite{DiestelUhl1977} and is still written $ \int_E v \; d \mathbb P$. If $E = \Omega$, then we write $\E(v) = \int_\Omega v \; d \mathbb P$. 

Given $1 < p < \infty$ the space of all Bochner measurable functions $v$ such that $\int_\Omega \|v(\omega)\|^p_\mathcal H \; d \mathbb P < \infty$ is denoted $\mathcal L^p(\Omega, \mathcal H)$. The space $L^p(\Omega, \mathcal H)$ is the quotient of $\mathcal L^p(\Omega, \mathcal H)$ by the space of zero $\mathbb P$-a.e. functions. $L^p(\Omega, \mathcal H)$ is a Banach space if $p \in [1, \infty)$ and is equipped with the norm $ \|v\|_{L^p(\Omega, \mathcal H)} = \left(\int_\Omega \|v(\omega)\|^p_\mathcal H \; d \mathbb P \right)^{1/p}$. If $p = 2$ and $\CH$ is a Hilbert space, then $L^2(\Omega, \mathcal H)$ is a Hilbert space with inner product $\aset{u,v}_{L^2(\Omega, \mathcal H)} := \int_\Omega \aset{u(\omega), v(\omega)}_{\mathcal H} \; d \mathbb P$ \cite{DiestelUhl1977}.










\begin{definition} \label{Definition of KLE of v}
Suppose $\mathcal H$ is a Hilbert space and $v$ is Bochner integrable. We will say that $v$ admits a KLE if 

\begin{equation} \label{KLE of v} 
    v = \E(v) + \sum_{r=1}^R \lambda_r^{1/2} Y_r \phi_r
\end{equation}

where the following properties hold 

\begin{enumerate}
    \item $R \in \N \cup \bset{\aleph_0}$  and if $R = \aleph_0$ the expansion \eqref{KLE of v} converges in the $L^2(\Omega, \mathcal H)$ norm. 


    \item $\lambda_r \in (0, \infty)$,  $\lambda_1 \geq \lambda_2 \geq \dots$ and $\sum_{r=1}^R \lambda_r = \|v - \E(v)\|^2_{L^2(\Omega, \CH)}$ 

    \item $\bset{Y_r}_{r=1}^R$ is a zero-mean orthonormal set in $L^2(\Omega)$ 

    \item $\bset{\phi_r}_{r=1}^R$ is an orthonormal set in $\mathcal H$


\end{enumerate}

\end{definition}

\subsection{The Space $\HS(\mathcal H, L^2(\Omega))$} \label{Background on Hilbert-Schmidt Operators}


From this point on, unless specified otherwise, $\CH$ denotes a Hilbert space. Let $\mathcal H, \mathcal G$ be two Hilbert spaces. Let $H \in \mathcal B(\mathcal H, \mathcal G)$. $H$ is said to be \textit{Hilbert-Schmidt} if

$$\sum_{e \in \mathscr E} \|He\|^2_\mathcal G < \infty$$

for arbitrary orthonormal basis (ONB) $\mathscr E$ of $\mathcal H$. Of course, this sum is defined iff $He = 0$ for all but at most countably many $e \in \mathscr E$. We denote the set of all Hilbert-Schmidt operators $\mathcal H \to \mathcal G$ by $\HS(\mathcal H, \mathcal G)$. It is routine to prove that $\HS(\mathcal H, \mathcal G)$ is a vector subspace of $\mathcal B(\mathcal H, \mathcal G)$ (and is proven so in \cite{Berberian2014}). Much of the current literature on Hilbert-Schmidt operators focuses on those which are self-adjoint and positive (\cite{Muger2022,kth,Shapiro2024}), endomorphisms on separable Hilbert spaces (\cite{kth,WangUSTC,Rosenweig,Melrose2018,Shapiro2024}), or integral operators (\cite{Garrett2012}). Eventually, we will produce a natural isomorphism between $L^2(\Omega, \CH)$ and $\HS(\CH, L^2(\Omega))$, but in order to do so, we must generalize the properties of Hilbert-Schmidt operators here. 


 \begin{theorem} \label{HS opearators are compact}
     $\mathrm{HS}(\mathcal H, \mathcal G)$ is a vector subspace of $\mathcal B_0(\mathcal H, \mathcal G)$ where the latter denotes the space of compact linear operators $\mathcal H \to \mathcal G$. 
 \end{theorem} 

 \begin{proof} 
Let $H \in \HS(\mathcal H, \mathcal G$), let $\mathscr E$ be an ONB of $\mathcal H$. If $He \neq 0$ for only finitely many $e \in \mathscr E$, then $H$ is finite rank and therefore compact. For the more interesting case, suppose $\bset{e_n}_{n=1}^\infty$ is an enumeration of the unit vectors in $\mathscr E$ for which $He \neq 0$. For each $n \in \N$ define $H_N$ by $H_Nx = \sum_{n=1}^N \aset{x, e_n}_\CH He_n$ for all $x \in \CH$. Each $H_N$ is of finite rank, its image lying in $\Span{He_n}_{n=1}^N$. We demonstrate that $\lim_{N\to \infty} \|H - H_N\|_\mathcal G = 0$ ($\|\cdot\|_\mathcal G$ denotes the operator norm on $\mathcal B(\mathcal H, \mathcal G)$) as follows: 

Let $x \in \CH$ be a unit vector. We have 

\begin{align*}
	\|(H - H_N)x\|_\mathcal G &= 
	\left \|	\sum_{n > N} \aset{x, e_n}_\CH He_n  \right \|_\mathcal G
	\\&\leq
	\left( \sum_{n > N} \aset{x, e_n}_\CH^2 \right)^{1/2} \left( \sum_{n >N}\| He_n \|_\mathcal G^2 \right)^{1/2}
	\\&\leq
	\left( \sum_{n >N}\| He_n \|_\mathcal G^2 \right)^{1/2}
\end{align*}

The final term goes to zero independently of $x$ as $N \to \infty$. Thus, $\bset{H_N}_{n = 1}^\infty$ is a sequence of finite rank operators approximating $H$ in the operator norm, which makes $H$ compact. 
 \end{proof}

By the spectral theorem, we can write the singular value decomposition (SVD) of each $H \in \HS(\mathcal H, \mathcal G)$ as $H = \sum_{r=1}^R \sigma_r g_r \widehat \otimes h_r$ where $R = \text{rank}(H)$; $\bset{\sigma_r}_{r=1}^R$ is a decreasing, positive, sequence which approaches 0 if $R = \aleph_0$; $\bset{g_r}_{r=1}^R$ is an orthonormal set in $\mathcal G$; and $\bset{h_r}_{r=1}^R$ is an  orthonormal set in $\CH$. Here $\widehat \otimes$ denotes the outer product of two vectors, i.e. $g \widehat \otimes h$ is defined to be the rank one operator $\mathcal H \to \mathcal G$ given by $x \mapsto \aset{x,h}_\mathcal H g$ for all $x \in \CH$. The SVD of $H$ allows us to make the following claim: 

\begin{lemma} \label{HS norm is independent of ONB}
Let $H \in \HS(\mathcal H, \mathcal G)$. Then for any two ONB $\mathscr E$, $\mathscr F$ of $\CH$, $\sum_{e \in \mathscr E} \|He\|_\mathcal G^2 = \sum_{f \in \mathscr F} \|Hf\|_\mathcal G^2$
\end{lemma}

\begin{proof}
	Let $\sum_{r=1}^R \sigma_r g_r \widehat \otimes h_r$ be the SVD of $H$. The statement of Lemma \ref{HS norm is independent of ONB} holds if, for arbitrarily chosen ONB $\mathscr E$ of $\CH$, $\sum_{r=1}^R \sigma_r^2  = \sum_{e \in \mathscr E} \|He\|_\mathcal G^2$. Let 
	
	\begin{align*}
		\mathscr E_0 &:= \bset{e \in \mathscr E: \aset{e,h_r}_\CH \neq 0 \text{ for some } h_r}
		\\&= \bigcup_{r=1}^R \bset{e \in \mathscr E: \aset{e, h_r}_\CH \neq 0}
	\end{align*}
	
	$\mathscr E_0$ is countable. We have 
	
	\begin{align*}
		\sum_{e \in \mathscr E} \|He\|_\mathcal G^2
		&= 
		\sum_{e \in \mathscr E} \left \|\sum_{r=1}^R \sigma_r \aset{e, h_r}_\CH g_r \right \|_\mathcal G^2
		\\&= 
		\sum_{e \in \mathscr E_0} \sum_{r=1}^R \sigma_r^2 \aset{e, h_r}_\CH^2 
			\\&= 
	\sum_{r=1}^R \sum_{e \in \mathscr E_0}  \sigma_r^2 \aset{e, h_r}_\CH^2 
		\\&= 
	\sum_{r=1}^R   \sigma_r^2  
	\end{align*}

	where the interchange of sums is justified since the terms $\sigma_r^2 \aset{e, h_r}_\CH^2 $ are non-negative. 

	
\end{proof}

Lemma \ref{HS norm is independent of ONB} demonstrates that the value of $\sum_{e \in \mathscr E} \|He\|^2_\mathcal G$ does not depend on the ONB $\mathscr E$. Thus we may unambiguously define the \textit{Hilbert-Schmidt norm} of a Hilbert-Schmidt operator as $\|H\|_{\HS(\mathcal H, \mathcal G)} :=  \left( \sum_{e \in \mathscr E} \|He\|^2_\mathcal G \right)^{1/2}$ for any choice of orthonormal basis $\mathscr E$ of $\mathcal H$. This norm can be easily seen to be generated by the Hilbert-Schmidt inner product $\aset{H_1, H_2}_{\HS(\mathcal H, \mathcal G)} = \sum_{e \in \mathscr E} \aset{H_1e, H_2e}_\mathcal G$.

\begin{corollary}
    If $H \in \mathrm{HS}(\mathcal H, \mathcal G)$, then $\|H\|^2_{ \HS(\mathcal H, \mathcal G)}$ equals the sum of the squares of the singular values of $H$. 
\end{corollary}

The following properties are proven in \cite{Berberian2014}

\begin{proposition} \label{properties of the hilbert-schmidt inner product}
$\empty$

    \begin{enumerate}
        \item $\aset{H_1, H_2}_{\HS(\mathcal H, \mathcal G)} := \sum_{e \in \mathscr E} \aset{H_1e, H_2e}_\mathcal G$ defines an inner product on $\HS(\mathcal H, \mathcal G)$ which generates the Hilbert-Schmidt norm. 

        \item $\HS(\mathcal H, \mathcal G)$ is complete w.r.t. this inner product.


    \end{enumerate}
\end{proposition}

With the above information, we can identify $L^2(\Omega, \mathcal H)$ with a subspace of $\HS(\mathcal H, L^2(\Omega))$ from Theorem \ref{v induces a HS operator} below. We will soon see that in fact $L^2(\Omega, \mathcal H)$ is isometrically isomorphic to $\HS(\mathcal H, L^2(\Omega))$.

\begin{theorem} \label{v induces a HS operator}
If $v \in L^2(\Omega, \mathcal H)$, then the operator $H_v: \mathcal H \to L^2(\Omega)$ defined by $(H_vx)(\omega) = \aset{v(\omega), x}_\mathcal H$ is Hilbert-Schmidt.
\end{theorem}

\begin{proof}
    As $v$ is ESV, we may, if necessary, redefine $v$ on a null set such that $v(\Omega)$ is separable. Put $\bset{e_n}_{n=1}^N$ an orthonormal basis of $\mathcal V := \overline{\Span{v(\Omega)}}$ where $N = \dim(\mathcal V) \in \N \cup \bset{\aleph_0}$. Let $\mathscr E$ be an extension of $\bset{e_n}_{n=1}^N$ to an ONB of $\mathcal H$. (This could be done by letting $\mathscr E'$ be an ONB of $\mathcal V^\perp$ and then letting $\mathscr E = \bset{e_n}_{n=1}^N \cup \mathscr E'$.)

   By the weak measurability of $v$ and the Cauchy-Schwarz inequality, $H_v$ does actually map each element of $\mathcal H$ to an element of $L^2(\Omega)$. Hence we have

   \begin{align*}
       \sum_{e \in \mathscr E} \|H_ve\|_{L^2(\Omega)}^2
       &= 
       \sum_{n=1}^N \|\aset{v, e_n}_\mathcal H\|^2_{L^2(\Omega)}
       \\&= 
        \sum_{n=1}^N \intp{\aset{v(\omega), e_n}^2_\mathcal H}
        \\&= 
        \intp{\sum_{n=1}^N \aset{v(\omega), e_n}^2_\mathcal H}
         \\&= 
        \intp{ \|v(\omega)\|^2_\mathcal H}
   \end{align*}
   
  If $N = \aleph_0$, then the interchange of integration and summation is justified by the Dominated Convergence since Bessel's inequality implies $\sum_{i=1}^n \aset{v(\omega), e_i}_\mathcal H^2 \leq \|v(\omega)\|^2_\mathcal H$ $\mathbb P-$a.e. for all $n \in \N$ and $\lim_{n \to \infty} \sum_{i=1}^n \aset{v(\omega), e_i}_\mathcal H^2 = \|v(\omega)\|^2_\mathcal H$ $\mathbb P-$a.e. Thus $\|H_v\|_{\HS(\mathcal H, L^2(\Omega))} = \|v\|_{L^2(\Omega, \CH)} < \infty$. 
\end{proof}

\begin{corollary} \label{Bochner norm equals HS norm equals sum of svsq}
    The correspondence $v \mapsto H_v$ is an isometry from $L^2(\Omega, \CH)$ to $\mathrm{HS}(\CH, L^2(\Omega))$. 
\end{corollary}

\subsection{Properties of the KL-Expansion}
\label{Properties of the KL-Expansion}

\begin{theorem} \label{v has a KLE}
    Suppose $v \in L^2(\Omega, \mathcal H)$, then $v$ has a KLE. 
\end{theorem}

\begin{proof}
    Define $v_0 := v - \E(v)$ and $H_{v_0}$ as in Theorem \ref{v induces a HS operator}. Let $H_{v_0}$ have SVD $\sum_{r=1}^R \lambda_r^{1/2} Y_r \widehat \otimes \phi_r$, with $\{ \lambda_r^{1/2} \}_{r=1}^R \subset \R$ the singular values, $\bset{Y_r}_{r=1}^R \subset L^2(\Omega)$ the left singular vectors and $\bset{\phi_r}_{r=1}^R \subset \mathcal H$ the right singular vectors. Now let $v' := \sum_{r=1}^R \lambda_r^{1/2} Y_r \phi_r$. We claim that $v'$ defines an element of $L^2(\Omega, \mathcal H)$. 

    Since $v'$ takes all its values in $\Span{\phi_r}_{r=1}^R$, it's ESV. As $\aset{v',x}_\mathcal H = \sum_{r=1}^R \lambda_r^{1/2} Y_r \aset{\phi_r, x}_\mathcal H$ describes an $L^2(\Omega)$-convergent sum of measurable random variables for any $x \in \mathcal H$, $v'$ is weakly measurable. Thus $v'$ is Bochner measurable by Pettis' Theorem. Finally, the orthonormality of $\bset{Y_r}_{r=1}^R$ and $\bset{\phi_r}_{r=1}^R$ guarantee that the $L^2(\Omega, \mathcal H)$ norm of $v'$ equals $\sum_{r=1}^R \lambda_r = \|H_{v_0}\|^2_{\text{HS}(\mathcal H, L^2(\Omega))} = \|v_0 \|^2_{L^2(\Omega, \mathcal H)} < \infty$.

    Now consider the Hilbert-Schmidt operator $H_{v'}$. For any $x \in \mathcal H$, we have 

    \begin{align*}
        H_{v'}x &= \sum_{r=1}^R \lambda_r^{1/2} Y_r \aset{ \phi_r,x }_\mathcal H
        \\&=
        \sum_{r=1}^R \lambda_r^{1/2} (Y_r \widehat \otimes \phi_r)x 
        \\&= H_vx 
    \end{align*}

    By Cor. \ref{Bochner norm equals HS norm equals sum of svsq}, the correspondence $v \mapsto H_v$ is an isometry; we must have $v_0 = v'$, which is to say $v = \E(v) + \sum_{r=1}^R \lambda_r^{1/2} Y_r \phi_r$.  
\end{proof}

\begin{corollary} \label{surjectivity}
     The correspondence $v \mapsto H_v$ represents an isometric isomorphism between $L^2(\Omega, \mathcal H)$ and $\mathrm{HS}(\mathcal H, L^2(\Omega))$
\end{corollary}

\begin{corollary}
    Both $L^2(\Omega, \mathcal H)$ and $\mathrm{HS}(\mathcal H, L^2(\Omega))$ are isometrically isomorphic to $L^2(\Omega) \otimes \mathcal H$. 
\end{corollary}

\begin{proof}
    Thm. \ref{v induces a HS operator} implies that the set of elements of $L^2(\Omega, \mathcal H)$ of the form $Xe$, $X \in L^2(\Omega)$, $e \in \mathcal H$ is total in $L^2(\Omega, \mathcal H)$. As in Sec. \ref{Background on Hilbert-Schmidt Operators}, the set of elements of $\text{HS}(\mathcal H, L^2(\Omega))$ of the form $X \widehat \otimes e$ is total in $\text{HS}(\mathcal H, L^2(\Omega))$. Furthermore, the mapping from $L^2(\Omega) \times \mathcal H$ to $L^2(\Omega, \mathcal H)$ given by $(X,e) \mapsto Xe$ is bilinear, as is the mapping from $L^2(\Omega) \times \mathcal H$ to $\text{HS}(\mathcal H, L^2(\Omega))$ given by $(X,e) \mapsto X \widehat \otimes e$. 
    
    Now let $X_1, X_2 \in L^2(\Omega)$, $e_1, e_2 \in \CH$. We have 

    \begin{align*}
        \aset{X_1e_1, X_2e_2}_{L^2(\Omega, \CH)} &= 
        \intp{\aset{X_1(\omega)e_1, X_2(\omega) e_2}_\CH }
        \\&= 
        \aset{e_1, e_2}_\CH \aset{X_1, X_2}_{L^2(\Omega)}
    \end{align*}

    Furthermore, for any arbitrary ONB $\mathscr E$ of $\mathcal H$ we have  

    \begin{align*}
        & \aset{X_1 \widehat \otimes e_1, X_2 \widehat \otimes e_2}_{\text{HS}(\CH, L^2(\Omega))} 
       \\  &= 
       \sum_{e \in \mathscr E} 
         \aset{(X_1 \widehat \otimes e_1)e, (X_2 \widehat \otimes e_2)e}_{L^2(\Omega)} 
         \\&= 
          \sum_{e \in \mathscr E} \aset{e_2, e}_\CH \aset{e_1, e}_\CH  
         \aset{X_1, X_2}_{L^2(\Omega)} 
         \\&= 
        \aset{e_1, e_2}_\CH \aset{X_1, X_2}_{L^2(\Omega)}
    \end{align*}

    Thus by \cite{Berberian2014}, we have $L^2(\Omega, \mathcal H) \cong \text{HS}(\mathcal H, L^2(\Omega)) \cong L^2(\Omega) \otimes \mathcal H$
\end{proof}

    A word of caution is in order. The tensor product of two Hibert spaces $\mathcal H \otimes \mathcal G$ is often defined to be the completion of the \textit{algebraic tensor product} -- i.e. the vector space \textit{finite} linear combinations of the monomials $h \otimes g$, $h \in \mathcal H$, $g \in \mathcal G$ -- with respect to the inner product $\aset{h_1 \otimes g_1, h_2 \otimes g_2} := \aset{h_1,h_2}_\mathcal H \aset{g_1, g_2}_\mathcal G$ \cite{alexanderian2013}. While $\mathcal H \otimes \mathcal G$ is a valid Hilbert space, $\otimes$ doesn't represent a categorical tensor product, in that continuous bilinear maps $T: \mathcal H \times \mathcal G \to \mathcal K$ do not always factor uniquely through $\mathcal H \otimes \mathcal G$ (an example is provided in Appendix \ref{The Failure of the Categorical Tensor Product}). 

As the KLE is derived from an SVD of a compact operator, we obtain properties 1-4 as enumerated in Definition \ref{Definition of KLE of v} immediately. We see that $R \in \N \cup \bset{\aleph_0}$ equals the rank of $H_{v_0}$. The only thing to check is that each $Y_r$ has mean zero. We invoke Hille's Theorem to do so (available as Thm. 6, Ch.2 of \cite{DiestelUhl1977})

\begin{theorem} [Hille]
    Let $v: \Omega \to \mathcal H$ be Bochner integrable. Let $T: \mathcal H \to \mathcal G$ be a closed linear operator between Banach spaces such that $Tv: \Omega \to \mathcal G$ is also Bochner integrable, then $T \int_E v \; d \mathbb P = \int_E \;  Tv \; d \mathbb P$. 
\end{theorem}

By Hille's Theorem, setting $T = \aset{\lambda_r^{-1/2} \phi_r, \cdot}$

\begin{align*}
\intp{Y_r} &= 
\intp{\aset{\lambda_r^{-1/2}\phi_r, v - \E(v)}_\mathcal H}
\\&= \lambda_r^{-1/2} \aset{\phi_r,  
\underbrace{
\intp{
v - \E(v)
}
}_{= 0}
}_\mathcal H 
\\&= 0 
\end{align*}

It is worth noting that the membership of $v$ in $L^2(\Omega, \CH)$ is not just sufficient for $v$ to admit a KLE, it is necessary. Indeed, any $v_0$ of the form $\sum_{r=1}^R \lambda_r^{1/2} Y_r \phi_r$ with properties 1-4 of Definition \ref{Definition of KLE of v} is ESV, weakly measurable, and square summable. Hence the function $v = v_0 + \E(v)$ belongs to $L^2(\Omega, \CH)$. 

Perhaps one of the most important properties of the KLE is the optimal $M$-term truncation property as follows: 

\begin{theorem}[Optimal $M$-term Truncation Property] \label{Optimal M term truncation Property}
	Let $v \in L^2(\Omega, \CH)$ have \eqref{KLE of v} as its KLE, and let $v_0 := v - \E(v)$. Then for any subspace $\mathcal S \subseteq \mathcal H$ of dimension $M \leq R$, $$\|(\id_\CH - \mathbf P_\mathcal S)v_0 \|_{L^2(\Omega, \CH)}^2 \geq \sum_{r > M} \lambda_r $$ with equality iff $\mathcal S = \Span{\phi_r}_{r=1}^M$. 
\end{theorem}

It is this property which allows for low dimensional computation with minimal error propagation. We postpone the proof for Section \ref{Naturality and Finite Dimensional Truncation}.

We conclude this section by mentioning that the non-separability assumption on $\mathcal H$ is only a mild generalization. The necessity that $v$ be ESV to possess a KLE allows us to redefine $v$ on a null set so that its codomain is separable, hence we can appeal to the previously developed theory behind KLEs. Additionally, non-separable Hilbert spaces are rare in applications and impossible in computations (only finite dimensional Hilbert spaces are admissible!). We include this content on non-separable spaces anyways for the interested reader and for potential undiscovered applications.

\subsection{Borel Measurability}

Since $\mathcal H$ constitutes a topogical space, a question of interest arises: What if we merely assume $v: \Omega \to \mathcal H$ is Borel measurable, i.e. $v^{-1}(E) \in \mathcal F$ for each Borel (with respect to the norm topology) subset $E \subseteq \mathcal H$? The uninterested reader may pass over this section. 

Assuming $v$ is Borel measurable, since $\aset{\cdot, e}$ is a continuous linear functional on $\mathcal H$ for any $e \in \mathcal H$, then clearly $v$ is weakly measurable. If $\mathcal H$ is separable, then $v(\Omega)$ is, of course, separably valued, so $v$ is Bochner measurable and thus lies in $L^2(\Omega, \mathcal H)$. However, if $\mathcal H$ is not separable, then the set of Borel measurable functions $v: \Omega \to \mathcal H$ may not even be a vector space! In particular, the sum of two Borel measurable functions may not be Borel measurable. 

\begin{example} Two Borel-measurable functions whose sum is not Borel. \end{example}
    Let $I$ be a set of cardinality greater than $\mathfrak c$, endowed with the discrete topology. Consider the (uncountable) standard orthonormal basis $\bset{e_i}_{i \in I}$ of $\ell^2(I)$\footnote{For any set $I$, $\ell^2(I)$ (resp. $\ell^2(I, \R^Q)$) denotes the Hilbert space of functions $v:I \to \R$ (resp. $v: I \to \R^Q$) such that $\sum_{i \in I} v(i)^2 < \infty$ (resp. $\sum_{i \in I} v(i)^\top v(i) < \infty$). Of course, $v$ must be non-zero for at most countably many $i$ in order that the sum make sense. This space possesses the inner product $\aset{u,v}_{\ell^2(I)} := \sum_{i \in I} u(i)v(i)$ (resp. $\aset{u,v}_{\ell^2(I, \R^Q)} := \sum_{i \in I} u(i)^\top v(i)$) }.

    Let $\Omega = I \times I$, and let $\mathcal F = \mathscr B(I) \otimes \mathscr B(I)$ where $\mathscr B(I)$ is the Borel $\sigma$-algebra on $I$. We define $v_1: I \times I \to \ell^2(I)$ by $v_1(\omega_1, \omega_2) = e_{\omega_1}$ and similarly $v_2(\omega_1, \omega_2) = -e_{\omega_2}$. We compute the pre-image of 0 $$(v_1 + v_2)^{-1}(\bset{0}) = \bset{(i, i) \in \Omega: i \in I} $$

    That the diagonal $ \bset{(i, i) \in \Omega: i \in I}$, is not in $\mathscr B(I) \otimes \mathscr B(I)$ when $\# I > \mathfrak c$ is the statement of Nedoma's pathology, a proof of which is available in \cite{schechter1996}. Thus, $v_1 + v_2$ is not Borel-measurable, even though $v_1$ and $v_2$ individually are.

Theorem 1.2 of Koumillis \cite{koumoullis1981on}  states that any Lebesgue measurable function $v:[0,1] \to \mathcal H$ is ESV (here $\mathcal H$ need merely be a metric space, not necessarily a Hilbert space). 
Lemma \ref{measurable functions on separable metric spaces are ESV} below generalizes Koumillis' statement for any $\Omega$ which is a separable metric space. 

\begin{comment}
As a corollary, any Lebesgue measurable function $v:\R \to \mathcal H$ is also ESV. If we define $\sigma: (0,1) \to \R$ by $\sigma(x) = \dfrac{e^x}{1 + e^x}$, then $v = (v \circ \sigma^{-1}) \circ \sigma$. The function $u := v \circ \sigma^{-1}$ is a Lebesgue measurable function $(0,1) \to \mathcal H$, and thus can be viewed as the restriction of a Lebesgue measurable function from $[0,1] \to \mathcal H$. Thus there exists a null set $\Omega_0 \subseteq (0,1)$ such that $u(\Omega_0^C)$ is separable. Since $v(\Omega_0^C) = u( \sigma^{-1}(\Omega_0^C))$, then $v$ will be ESV if $\sigma^{-1}(\Omega_0^C)$ has Lebesgue measure zero. This holds because 

$$\lambda(\sigma^{-1}(\Omega_0^C)) = \int_{\sigma^{-1}(\Omega_0^C)} dx = \int_{\Omega_0} (\sigma^{-1})'(x) dx = 0$$

where $\lambda$ denotes the Lebesgue measure and $(\sigma^{-1})'(x) = \sigma(x)(1 - \sigma(x))$ denotes the ordinary derivative of $(\sigma^{-1})$. More generally, if we consider ANY $(\Omega, \mathcal F, \mathbb P)$ where $\Omega$ is a separable metric space and $\mathcal F$ contains the Borel $\sigma$-algebra on $\Omega$, then any Borel-measurable function $v: \Omega \to \mathcal H$ is ESV. 
\end{comment}

 \begin{lemma} 
 \label{measurable functions on separable metric spaces are ESV} If $(\Omega, \mathcal F, \mathbb P)$ is such that $\Omega$ is a separable metric space and $\mathcal F$ contains the Borel sets of $\Omega$, then if $v: \Omega \to \mathcal H$ is measurable, $v$ is ESV. 
 \end{lemma}

\begin{proof}
    Mimicking the proof of Lemma 1.1.12 in Hytonen \cite{hytonen2016}, we AFSOC that  $v(\Omega_1)$ is non-separable for all $\Omega_1 \in \mathcal F$ of measure one. Then there exists an uncountable, disjoint collection of non-empty open sets $\bset{O_i}_{i\in I} $ each intersecting $v(\Omega_1)$. For each subset $J$ of $I$, define $O_J := \bigcup_{i \in J} O_i$. For any $J \subseteq I$, $v^{-1}(O_J)$ is a Borel subset of $\Omega$. For all $J \neq J' \subseteq I$, $v^{-1}(O_J) \neq v^{-1}(O_{J'})$. So there must be at least $2^{\#I}$ Borel subsets of $\Omega_1$, but since $\Omega_1$ is separable, there can only be $2^{\aleph_0}$ Borel subsets. 
\end{proof}

Note that as in \cite{koumoullis1981on}, $\CH$ in Lemma \ref{measurable functions on separable metric spaces are ESV} need only be a metric space instead of a Hilbert space. 


\begin{comment}
 \begin{corollary}
 	For any $E \in \mathcal F$, let $E \cap \mathcal F := \bset{E \cap F: F \in \mathcal F}$. 
    If $(\Omega, \mathcal F, \mathbb P)$ is any complete probability space with a set $\Omega_1 \in \mathcal F$ such that $\mathbb P(\Omega_1) = 1$, $\# \Omega_1 \cap \mathcal F = 2^{\aleph_0}$, and $v: \Omega \to \mathcal H$ is Borel, then $v$ is ESV. 
\end{corollary}
\end{comment}

\begin{definition} [Tight Measure] Let $\Omega$ be a metric space and $\mathbb P$ a Borel measure on $\Omega$. $\mathbb P$ is tight if for every $\varepsilon > 0$, there exists a compact subset $K_\varepsilon \in \mathcal F$ such that $\mathbb P (\Omega \setminus K_\varepsilon) < \varepsilon$.
\end{definition}

\begin{proposition}
    Let $\mathbb P$ be a tight Borel measure on the metric space $\Omega$, and $v: \Omega \to \mathcal H$ be Borel measurable. Then $v$ is ESV.
\end{proposition}

\begin{proof}
    For each $n \in \N$, let $K_n$ be a compact set in $\mathcal F$ such that $\mathbb P(\Omega \setminus K_n) < \frac 1 n$. Let $K_\infty = \bigcup_{n \in \N} K_n$. $\mathbb P(K_\infty) = 1$ and is a separable metric space ($\sigma$-compact even) and $K_\infty \cap \mathcal F := \bset{K_\infty \cap E: E \in \mathcal F}$ contains the Borel $\sigma$-algebra on $K_\infty$. By Lemma \ref{measurable functions on separable metric spaces are ESV}, $v(K_\infty)$ is ESV. 
\end{proof}

\begin{comment}
If we assume that $\mathbb P$ is $\tau$-additive instead of simply countably additive, then we get ... $\mathbb P$ is said to be $\tau$-additive if whenever $\bset{E_\alpha}_{\alpha \in A}$ is a monotonic (decreasing) net of measurable subsets of $\Omega$, then $\lim_\alpha \Pr(E_\alpha) = \Pr \left( \cap_{\alpha \in A} E_\alpha \right)$. In particular, if each $\bset{E_\alpha}_{\alpha \in A}$ is decreasing and has measure 1, then $\Pr \left( \cap_{\alpha \in A} E_\alpha \right) = 1$. By passing to complements if $\bset{E_\alpha}_{\alpha \in A}$ is \textit{increasing} and has measure 0, then $\Pr \left( \cup_{\alpha \in A} E_\alpha \right) = 0$.
\end{comment}

\subsection{Naturality and Finite Dimensional Computation} \label{Naturality and Finite Dimensional Truncation}

Consider the category \textbf{Hilbert} of Hilbert spaces over $\R$ whose morphisms are norm-continuous linear operators. We see that for each measure space $(\Omega, \mathcal F, \mathbb P)$ the associations 

$$
    \mathcal H \mapsto L^2(\Omega, \mathcal H), \qquad
    \mathcal H \mapsto \text{HS}(\mathcal H, L^2(\Omega))
$$

are functorial over \textbf{Hilbert}. We denote these functors $L^2(\Omega, -)$ and $\text{HS}(-, L^2(\Omega))$ respectively. 

Let

$$\mathcal G \xrightarrow{S} \mathcal H \xrightarrow{T} \mathcal K$$

be a diagram of Hilbert spaces.The functor $L^2(\Omega, -)$ should send an operator $T: \mathcal H \to \mathcal K$ to the operator $\tilde T: L^2(\Omega, \mathcal H) \to L^2(\Omega, \mathcal K)$ where $\tilde T$ sends the function $\omega \mapsto v(\omega)$ in $L^2(\Omega, \mathcal H)$ to the function $\omega \mapsto T \circ v(\omega)$ in $L^2(\Omega, \mathcal K)$. We will abuse notation and use $T$ to describe both the operator $T: \mathcal H \to \mathcal K$ AND $\tilde T: L^2(\Omega, \mathcal H) \to L^2(\Omega, \mathcal K)$. Perhaps less obviously, the functor $\text{HS}(-,L^2(\Omega))$ must send the operator $T: \mathcal H \to \mathcal K$ to the map 
$ (- \circ T^*): \text{HS}(\mathcal H, L^2(\Omega)) \to \text{HS}(\mathcal K, L^2(\Omega))$, where we define $(- \circ T^*)H := HT^*$. That is, $(- \circ T^*)$ is the ``post-compose with $T^*$" operator on $\text{HS}(\mathcal H, L^2(\Omega))$. It is routine to show that $Tv \in L^2(\Omega, \mathcal K)$, $HT^* \in \text{HS}(\mathcal K, L^2(\Omega))$ and that these functors respect the composition law.

\begin{proposition}
    The correspondence $v \mapsto H_v$ is a \textit{natural} isomorphism between the functors $L^2(\Omega, -)$ and $\mathrm{HS}(-,L^2(\Omega))$. 
\end{proposition}

Let $\iota$ be this correspondence. To say that this correspondence is a \textit{natural} isomorphism is to say that $\iota$ is an isomorphism of Hilbert spaces (which we've stated as Corollary \ref{surjectivity}), and that for every  $T \in \mathcal B(\mathcal H, \mathcal K)$, we get a commutative square 

\[
\begin{tikzcd}
L^2(\Omega, \mathcal H) 
\arrow[r, "T"] \arrow[d, "\mathbf \iota"] & 
L^2(\Omega, \mathcal K) \arrow[d, "\iota"] \\
\text{HS}(\mathcal H, L^2(\Omega) ) \arrow[r, "(-\circ T^*)"] & 
\text{HS}(\mathcal K, L^2(\Omega) )
\end{tikzcd}
\]

\begin{proof}
We verify the details with a ``diagram chase", i.e. we must verify that for every $v \in L^2(\Omega, \mathcal H)$ and every $k \in K$ we have $\iota(Tv)k = ((-\circ T^*)\iota(v))k$

\begin{align*}
    \iota(Tv)k &= H_{Tv}k
    \\&= 
    \aset{Tv, k}_{L^2(\Omega)}
    \\&= 
    \aset{v, T^*k}_{L^2(\Omega)}
    \\&=
    H_v(T^*k)
    \\&=
    (-\circ T^*)H_vk
    \\&=
    ((-\circ T^*)\iota(v))k
\end{align*}
\end{proof}

The importance of this natural isomorphism arises in computational purposes. In computation, we cannot work with infinite dimensional Hilbert spaces. In less formal words, we would like to ensure that, given a finite dimensional subpace $\mathcal H_M$ of $\mathcal H$, that projecting $v \in L^2(\Omega, \mathcal H)$ onto $L^2(\Omega, \mathcal H_M)$ is ``the same" as starting off in $L^2(\Omega, \mathcal H_M)$. 

Likewise, when dealing with the operator $H_{v_0} \in \text{HS}(\mathcal H, L^2(\Omega))$, we want to ensure that projecting $H_{v_0}$ onto $\text{HS}(\mathcal H_M, L^2(\Omega))$ is ``the same" as starting off in $\text{HS}(\mathcal H_M, L^2(\Omega))$. These concerns are answered using the language of category theory that we've thus far developed. Here we put $\mathcal K = \mathcal H_M$ and $T = \mathbf P_{\mathcal H_M} = $ the projection onto $\mathcal H_M$ operator.  

It is not much harder to show that the isometric isomorphism $Xe \leftrightarrow X \otimes e$ induces a natural isomorphism between the functors $L^2(\Omega, -)$ and $L^2(\Omega) \otimes -$, which in this case is also an isometry. The language of natural isomorphisms allows a concise proof of the optimal $M$-term truncation property of KLEs as in Theorem \ref{Optimal M term truncation Property}.

\begin{proof}[Proof of Theorem \ref{Optimal M term truncation Property}]
	Let $v \in L^2(\Omega, \CH)$ have \eqref{KLE of v} as its KLE. Put $\mathcal Y = \Span{Y_r}_{r=1}^R$ and $\Phi = \Span{\phi_r}_{r=1}^R$. We first demonstrate that may first assume that $\mathcal S \subseteq \Phi$. By decomposing $\mathbf P_\mathcal S$ as $\mathbf P_{S \cap \Phi} + \mathbf P_{\mathcal S \cap \Phi^\perp}$ we have 
	
	$$\mathbf P_\mathcal S v_0 = \mathbf P_{S \cap \Phi}v_0 + \mathbf P_{\mathcal S \cap \Phi^\perp}v_0 = \mathbf P_{S \cap \Phi}v_0$$
	
	whence 
	
	$$\|(\id_\CH - \mathbf P_\mathcal S)v_0 \|_{L^2(\Omega, \CH)}^2 = \|(\id_\CH - \mathbf P_{\mathcal S \cap \Phi})v_0 \|_{L^2(\Omega, \CH)}^2$$

	But $\dim(\mathcal S \cap \Phi) = M$ iff $\mathcal S \subseteq \Phi$.  If $\mathbf P_{\mathcal S}: \CH \to \mathcal S$ is the projection onto $\mathcal S$ operator, then the naturality of the identification $\iota$ given by $Xe \leftrightarrow X \otimes e$ between $L^2(\Omega, -)$ and $L^2(\Omega) \otimes -$ implies
	
	\begin{align*}
	&\|v_0 - \mathbf P_{\mathcal S} v_0\|^2_{L^2(\Omega, \CH)} 
	\\&= 
    \|\iota(v_0) - \mathbf P_{L^2(\Omega) \otimes \mathcal S} \iota(v_0)\|^2_{L^2(\Omega) \otimes \CH}
	\\&= 
	\|\iota(v_0) - \mathbf P_{\mathcal Y \otimes \mathcal S} \iota(v_0)\|^2_{\mathcal Y \otimes \Phi}
	\end{align*}

	 Because $\bset{\lambda_r, Y_r, \phi_r}_{r=1}^R$ obeys properties 1-4 of Definition \ref{Definition of KLE of v}, and $\mathcal Y$ and $\Phi$ are separable Hilbert spaces, we may refer to Thm. 2.7 of \cite{Schwab2006}, which states that the minimizer $\mathcal S$ of $\|\iota(v_0) - \mathbf P_{\mathcal Y \otimes \mathcal S} \iota(v_0)\|^2_{\mathcal Y \otimes \Phi}$ over all subspaces of dimension $M$ is given by $\Phi_M := \Span{\phi_r}_{r=1}^M$. $\Phi_M$ must also be the minimizer of $\|v_0 - \mathbf P_{\mathcal S} v_0\|^2_{L^2(\Omega, \CH)}$ over all subspaces $\mathcal S$ of dimension $M$. Finally, it is straightforward to compute that $	\|v_0 - \mathbf P_{\Phi_M} v_0\|^2_{L^2(\Omega, \CH)} = \D \sum_{r > M} \lambda_r$.
	
	

\end{proof}

\begin{comment}
Thus the squares

\[
\begin{tikzcd}
\mathcal H 
\arrow[r, "L^2(\Omega{,}-)"] \arrow[d, "\mathbf P_{\mathcal H_M}"] & 
L^2(\Omega, \mathcal H) \arrow[d, "\mathbf P_{\mathcal H_M}"] \\
\mathcal H_M \arrow[r, "L^2(\Omega{,}-)"] & 
L^2(\Omega, \mathcal H_M)
\end{tikzcd}
\qquad
\begin{tikzcd}
\mathcal H 
\arrow[r, "\text{HS}(-{,}L^2(\Omega))"] \arrow[d, "\mathbf P_{\mathcal H_M}"] & 
\text{HS}(\mathcal H, L^2(\Omega)) \arrow[d, "\mathbf P_{\mathcal H_M}"] \\
\mathcal H_M \arrow[r, "\text{HS}(-{,}L^2(\Omega))"] & 
\text{HS}(\mathcal H_M, L^2(\Omega))
\end{tikzcd}
\]

commute. Furthermore, the correspondence $\iota(v) = H_v$, where $H_v$ is the operator $h \mapsto \aset{h,v}_{\mathcal H}$ is a natural isomorphism between these two functors. 

\[
\begin{tikzcd}
L^2(\Omega, \mathcal H) \arrow[r, "T"] \arrow[d, "\iota"] & 
L^2(\Omega, \mathcal K) \arrow[d, "\iota"] \\
\text{HS}(\mathcal H, L^2(\Omega)) \arrow[r, "(-\circ T^*)"] & 
\text{HS}(\mathcal K, L^2(\Omega))
\end{tikzcd}
\]

We readily check for each $v \in L^2(\Omega, \mathcal H)$, and for each $h \in \mathcal H$: 

\begin{align*}
    (\iota \circ T(v))h &= 
    H_{Tv}h
    \\&=
    \aset{h, Tv}_{\mathcal H}
    \\&=
    \aset{T^*h, v}_{L^2(\Omega)}
    \\&= 
   (H_vT^*)h
    \\&= 
     ((-\circ T^*)\iota(v))h
\end{align*}
\end{comment}
\section{Example: A Vector Field KLE (Japanese Mortality Data)}

We include an example where computing KLE in a general Hilbert space offers an advantage over the usual $L^2(\mathcal T)$ case. The Japanese mortality data, available at \cite{japan2023}, include the number of deaths per year in 47 prefectures of Japan from the years 1947-2023. The deaths are further subdivided by age, ranging from 0-110+ (individuals over 110 years of age are grouped together). As part of an experimental study on a dimensionality reduction pipeline, Gao et. al. (\cite{gao2017}) implement KL transform theory in $L^2(\mathcal T)$ and factor analysis on this dataset. (\textit{Remark}: At the time of publication, only the years 1947-2015 were available at \cite{japan2023}.) Gao et. al. treat the Japanese mortality data as a \textit{functional time series} - a sequence $\pmb{\mathcal X}_1, \dots \pmb{\mathcal X}_T$, the subscripted index corresponding to time. Here $\pmb{\mathcal X}_t = (\mathcal X_t^1(u), \dots \mathcal X_t^Q(u))$, $1 \leq t \leq T$, each $\mathcal X_t^i$ itself an element of $L^2(\mathcal T)$ where $\mathcal T $ is an interval of $\R$. For this particular data set $\mathcal T = [0, 110]$. The temporal index $t$ corresponds to the year 1947-2015, the spatial variable $u$ corresponds to age (0-110+), and the superscripted indices $1,...,Q$ correspond to each of the 47 prefectures.

In performing dimension reduction via the KLE, Gao et. al. ``ignore" the temporal significance of the subscripted index $t$, treating each $\pmb{\mathcal X}_t$ as a realization of a random element of the space $L^2(\mathcal T, \R^Q)$ - meaning for each $1 \leq t \leq T$ and for each $1 \leq q \leq Q$ $\mathcal X_t^q$ denotes a realization of an element of $L^2(\mathcal T)$, the space in which the KLE is typically practiced. From there, each $\mathcal X_t^q$ is replaced with the sum of its first $R_0$ KLE terms, for some suitable $R_0 \in \N$. 

In this set up, we may alternatively identify the functional time series $\pmb{\mathcal X}_1, \dots \pmb{\mathcal X}_T$ with $T$ realizations of an element $v \in L^2(\Omega, \mathcal H)$ where $\mathcal H = L^2(\mathcal T, \R^{47})$. For our example, we make two modifications: we reduce our scope from all 47 prefectures to the first 5 (Hokkaido, Aomori, Iwate, Miyagi, and Akita). We furthermore replace $L^2(\mathcal T, \R^Q)$ with $\mathcal H = \ell^2(\mathcal I, \R^5 )$ where $\mathcal I = \bset{0,1,\dots, 110}$ eliminating the need to account for the error incurred by approximating each mortality curve as a real-valued function sampled discretely at ${0,1, \dots 110}$. In this updated setup, we have $v(\omega, i) = (v^1(\omega,i), \dots v^5(\omega, i))$. Here $\omega \in \Omega$ denotes the stochastic variable and $i \in \mathcal I$ the ``spatial" variable, corresponding to age of mortality, and each $v^q$, $1 \leq q \leq 5$ describes an element of $L^2(\Omega, \ell^2(\mathcal I))$. Proceeding as Gao et. al., we compute for each $1 \leq q \leq 5$, the $R_0$th KL trucnation of the $q$th component $v^q$, which we denote $v^q_{R_0}$, giving 

\begin{align*}
    v^1_{R_0{}} &= \E(v^1) + \sum_{r=1}^{R_0} \lambda_{1r}^{1/2} Y_{1r} \phi_{1r} 
    \\ &\vdots \\
    v^5_{R_0} &= \E(v^5) + \sum_{r=1}^{R_0} \lambda_{5r}^{1/2} Y_{5r} \phi_{5r}  
\end{align*}

$\bset{\lambda_{qr}} \in \R$, $\bset{Y_{qr}} \in L^2(\Omega)$, $\bset{\phi_{qr}} \in \ell^2(\mathcal I)$. Hence we define

\begin{equation} \label{Component-wise KLE}
\begin{split}
    v_{R_0} &= 
\sum_{q=1}^5 \left( \E(v^q) + \sum_{r=1}^{R_0} \lambda_{qr}^{1/2} Y_{qr} \phi_{qr} \right) \mathbf e_q
\\&=
\E(v) + \sum_{q=1}^5  \sum_{r=1}^{R_0} \lambda_{qr}^{1/2} Y_{qr} (\phi_{qr} \mathbf e_q)
\end{split}
\end{equation}

where $\bset{\mathbf e_q}_{q=1}^5$ is the standard basis of $\R^5$. We will denote truncation \eqref{Component-wise KLE} as the \textit{component-wise KL truncation} of $v$. 

The collection $\bset{\phi_{qr}\mathbf e_q}$ form an ONB of an $(R_0 \times 5)$-dimensional subspace $\mathcal S$ of $\ell^2(\mathcal I, \R^5)$, hence \eqref{Component-wise KLE} represents an orthogonal projection of the centered $v$ onto $\mathcal S$. Now, let 

\begin{equation} \label{Vector-field KL}
\tilde v_{R_0} = \E(v) + \sum_{r=1}^{R_0 \times 5} \tilde{\lambda_r}^{1/2} \tilde{Y_r} \tilde{\phi_r}
\end{equation}

denote the $(R_0 \times 5)$th KL truncation of $v$, this time treating $v$ as an element of $L^2(\Omega, \ell^2(\mathcal I, \R^5))$. (\textit{Remark:} $\bset{\tilde{\phi_r}}$ describe elements of $\ell^2(\mathcal I, \R^5)$.) We will denote truncation \eqref{Vector-field KL} as the \textit{vector field KL truncation} of $v$. 

By Theorem \ref{Optimal M term truncation Property}, the mean-squared error incurred by truncation \eqref{Vector-field KL} must be less than that of \eqref{Component-wise KLE}, i.e. $\|v-\tilde v_{R_0}\|^2_{L^2(\Omega, \mathcal H)} \leq \|v-v_{R_0}\|^2_{L^2(\Omega, \mathcal H)}$. We confirm this with experiments, observing that the relative squared error of the component-wise KL truncation - i.e. the quotient $\dfrac{\|v - v_{R_0}\|^2_{L^2(\Omega, \mathcal H)} }{ \|v\|^2_{L^2(\Omega, \mathcal H)} }$ - is considerably less than that of the vector field KL truncation - i.e. the quotient $\dfrac{\|v - \tilde v_{R_0}\|^2_{L^2(\Omega, \mathcal H)} }{ \|v\|^2_{L^2(\Omega, \mathcal H)} }$ (see Figure \ref{JMD figure}). In fact the component-wise KL truncation requires 30 terms to obtain a similar relative squared error as a 5 term vector field KL truncation. Our results demonstrate that the component-wise expansion implemented by Gao in \cite{gao2017} is suboptimal in the $\|\cdot\|_{L^2(\Omega, \mathcal H)}^2$ sense.

\begin{figure}[H]
    \centering
\includegraphics[width=0.95\linewidth]{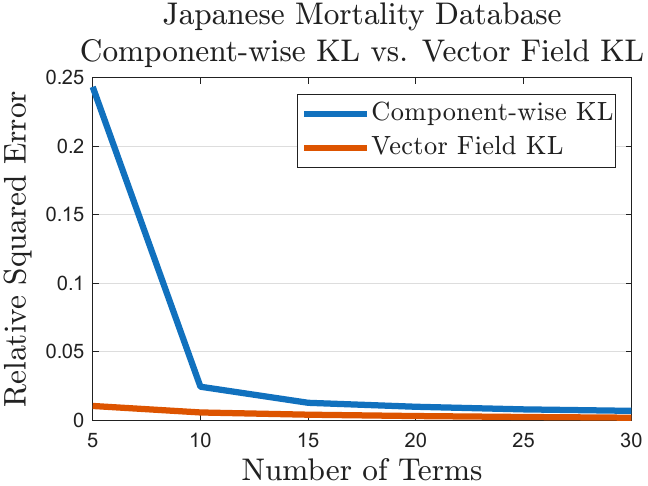}
    \caption{Comparison of the relative mean squared error incurred by the component-wise KL truncation and vector field KLE for six different values of the truncation $R_0$.}
    \label{JMD figure}
\end{figure}
\section{Conclusion}

The Karhunen-Lo\`eve expansion of a function $v: \Omega \to L^2(\mathcal T)$ is a well known method for decomposing a square-summable stochastic process into a countable number of simple components with provable optimality properties. While most of the generalizations of the KLE focus on square-summable functions of the form $v: \Omega \to \mathcal H$ where $\mathcal H$ is a Hilbert space of scalar or $\R^Q$-valued functions, the existence of an expansion for arbitrary - even non-separable - $\mathcal H$ can be demonstrated by appealing to the isomorphism between the spaces $L^2(\Omega, \mathcal H)$ and $\text{HS}(\mathcal H, L^2(\Omega))$ induced by the natural isomorphism between the functors $L^2(\Omega, -)$ and $\text{HS}(-, L^2(\Omega))$ allowing for a more complete, cohesive theory. We furthermore demonstrate that the separability assumptions usually placed on the image of $v$ isn't just a computational convenience, but a necessary condition for the KLE of $v$ to hold. The usefulness of such a general view of the KLE holds in the capability to represent a multitude data types corresponding to different Hilbert spaces, from lists, to time series, to images, to solutions of PDEs.

\appendix

\section{The Failure of the Categorical Tensor Product}
\label{The Failure of the Categorical Tensor Product}

We demonstrate in this section that the ``tensor product" as described in \ref{Properties of the KL-Expansion} fails to be a tensor product in the category \textbf{Hilbert}. We adopt the example from \cite{garrett2010}. Let $\CH$ be a separable Hilbert space with ONB $\bset{e_n}_{n=1}^\infty$ and suppose we have the commutative diagram as below

\begin{equation} \label{Tensor Product Commutative Diagram}
\begin{tikzcd} 
\CH \times \CH \arrow[r, "b"] \arrow[d, "\otimes"'] & \R \\
\CH \otimes \CH \arrow[ur, "\varphi"'] & 
\end{tikzcd}
\end{equation}

where $b$ is the continuous bilinear map $(u,v) \mapsto \aset{u,v}_\CH$, $\otimes$ is the bilinear map $(u,v) \mapsto u \otimes v$, and $\varphi$ is the linear map $u \otimes v  \mapsto \aset{u,v}_\CH$. 

Diagram \ref{Tensor Product Commutative Diagram} exists and commutes in the category of vector spaces over $\R$. However, when each object in \ref{Tensor Product Commutative Diagram} is endowed with its inner product structure, the maps $\otimes$ and $\varphi$ fail to be continuous. 

To demonstrate that $\otimes$ is not continuous, define for each $N \in \N$,  $T_N \in 
\CH \times \CH$ by $\sum_{n=1}^N \frac{1}{n} (e_n, e_n)$. While $ \lim_{N \to \infty} T_N \to \sum_{n=1}^\infty \tfrac 1 n (e_n, e_n)$, a well-defined element of $\CH \times \CH$, we have that $\lim_{N \to \infty} \|\varphi(T_N)\|_{\CH \otimes \CH} = \sum_{n=1}^\infty \tfrac 1 n = \infty$. 

Further, define for each $N \in \N$, $S_N \in \CH \otimes \CH$ by $S_N := \ \sum_{n=1}^N \frac 1 n e_n \otimes e_n$. Then $S_N$ converges to the well-defined element of $\CH \otimes \CH$ given by $\ \sum_{n=1}^\infty  \frac 1 n e_n \otimes e_n $. However, $\lim_{N \to \infty} |\varphi(S_N)| = \sum_{n=1}^\infty \frac 1 n = \infty$. 

\newpage 

\section{Notation} \label{Notation}

\begin{center}
\begin{tabular}{|C{4 em}|p{15 em}|}
	\hline
	Symbol & Meaning \\
	\hline
	$\E$ & Expectation \\
	$\mathcal H$, $\mathcal G$ & Hilbert spaces \\
	$\aset{\cdot, \cdot}_\mathcal H$ & Inner product on $\mathcal H$ \\
	$\|\cdot \|_\mathcal H$ & Norm on $\mathcal H$ \\
	$\mathcal S^\perp$ & Orthogonal complement of the subspace $\mathcal S$ \\
	$\Omega$ & Sample space \\
	$\mathcal F$ & $\sigma$-algebra of events \\
	$\mathbb P$ & Probability measure \\
	$L^2(X)$ & Set of (equivalence classes) of square summable, measurable  functions $X \to \R$ for some measure space $X$ \\
	$L^2(X, \R^Q)$ & Set of (equivalence classes) of square summable, measurable functions $X \to \R^Q$ \\
	$\mathbb I_E$ & Indicator function of $E$ \\
	$L^2(\Omega, \mathcal H)$ & Set of (equivalence classes) of square summable, Bochner-measurable functions $\Omega \to \mathcal H$ \\
	$\ell_2(X)$, $\ell_2(X, \R^Q)$ & $L^2(X)$, $L^2(X, \R^Q)$ respectively with counting measure \\
	$v$ & element of $L^2(\Omega, \mathcal H)$ \\
	$v_0$ & $v - \E(v)$ \\
	$\int_E v \; d \mathbb P$ & Bochner integral of $v$ over $E$ \\
	$\text{HS}(\mathcal H, \mathcal G)$ & Set of Hilbert-Schmidt operators $\mathcal H \to \mathcal G$ \\
	$H_v$ & Operator $\CH \to L^2(\Omega)$: $x \mapsto \aset{x,v}_\mathcal H$ \\
	$R$ & Rank of $H_v$ \\
	$\bset{\lambda_r}_{r=1}^R$ & Eigenvalues of $v$ \\
	$\bset{\phi_r}_{r=1}^R$ & Right singular vectors of $v$ \\
	$\bset{Y_r}_{r=1}^R$ & Stochastic components of $v$ \\
	$\mathcal B(\mathcal H, \mathcal G)$ & Set of bounded linear operators $\mathcal H \to \mathcal G$ \\
	$\mathcal B_0(\mathcal H, \mathcal G)$ & Set of compact linear operators $\mathcal H \to \mathcal G$ \\
	$\mathcal B_{00}(\mathcal H, \mathcal G)$ & Set of finite rank linear operators $\mathcal H \to \mathcal G$ \\
	$g \widehat{\otimes} h$ & Outer product of $g$ and $h$, a rank one operator $x \mapsto \aset{x,h}g$ \\
	$\mathbf e_q$ & $q$th standard basis vector of $\R^Q$ \\
	$\mathbf P_\mathcal S$ & Orthogonal projection onto the closed subspace $\mathcal S$ operator \\
	$\id_\mathcal H$ & Identity on $\mathcal H$ operator \\
	$\mathcal H \otimes \mathcal G$ & Hilbert space tensor product of $\mathcal H$ and $\mathcal G$ \\
	$\delta_{mn}$ & Kronecker delta function of $m$ and $n$ \\
	\hline
\end{tabular}
\end{center}

\vspace{60pt}
\printbibliography
\end{multicols*}

\end{document}